\documentclass{compositio}
\usepackage{mathrsfs}
\usepackage{amsmath}
\usepackage{amssymb}
\usepackage{verbatim}
\usepackage{xcolor}
\usepackage{graphicx}
\usepackage{textcomp}
\usepackage[all]{xy}
\usepackage{float}
\numberwithin{equation}{section}

\theoremstyle{plain}
\newtheorem{theorem}{Theorem}[section]
\newtheorem{proposition}{Proposition}[section]

\newtheorem{lemma}{Lemma}[section]

\theoremstyle{definition}

\theoremstyle{remark}

\begin{document}

\title{The balanced cone of the small resolution of the quintic conifold}

\author{Jixiang Fu}
\email{majxfu@fudan.edu.cn}
\address{School of Mathematical Sciences, Fudan University, Shanghai 200433, People's Republic of China}

\author{Hongjie Wang}
\email{21110180020@m.fudan.edu.cn}
\address{School of Mathematical Sciences, Fudan University, Shanghai 200433, People's Republic of China}

\thanks{Fu is supported by NFSC, grant 12141104.}

\begin{abstract}
	In this note, we use the intersection number to determine explicitly the balanced cone of the small resolution of the quintic conifold.
\end{abstract}
\maketitle

\section{Introduction}
For a compact Kähler manifold X, we can define its Kähler cone and balanced cone as follows.
\begin{equation*}
	\begin{aligned}
		\mathcal{K}(X) &= \{\alpha \in H^{1,1}(X,\mathbb{R}) \mid \alpha \text{ admits a Kähler metric representative}\}, \\
		\mathcal{B}(X) &= \{\alpha \in H^{n-1,n-1}(X,\mathbb{R}) \mid \alpha \text{ admits a balanced metric representative}\}.
	\end{aligned}
\end{equation*}

There exists a natural map between these two cones, known as the balanced map, denoted by
\begin{equation*}
	\mathbf{b} : \mathcal{K}(X) \rightarrow \mathcal{B}(X), \quad \alpha \mapsto \alpha^{n-1}.
\end{equation*}
Fu and Xiao \cite{fu2014relations} proved that the balanced map is injective but not necessarily surjective. In fact, they provided numerous examples whose balanced maps are not surjective, including the small resolution of the quintic conifold.

As an exercise, this paper characterizes the balanced cone of this manifold, thereby comparing the image of its K\"ahler cone under the balanced map with the balanced cone. This result makes the investigation of non-K\"ahler canonical metrics on the small resolution of the quintic conifold possible.

Denote by $Y$ the small resolution of the quintic conifold in $\mathbb{P}^4$, with the embedding $Y \subset \hat{\mathbb{P}}^4 \subset \mathbb{P}^4 \times \mathbb{P}^1$, where $\hat{\mathbb{P}}^4$ is the blow-up of $\mathbb{P}^4$ along the submanifold $\mathbb{P}^2$. In this paper, we establish the following result.  

\begin{theorem}  
	The balanced cone $\mathcal{B}(Y)$ of $Y$ is bounded by two rays generated by $\alpha \wedge \beta$ and $\beta \wedge \beta - \frac{1}{4} \alpha \wedge \beta$, respectively, where $\alpha$ and $\beta$ are the pullbacks to $Y$ of the K\"ahler classes associated with the Fubini-Study metrics on $\mathbb{P}^1$ and $\mathbb{P}^4$, respectively.  
\end{theorem}

\begin{figure}[H]
	\centering
	\includegraphics[width=0.55\textwidth]{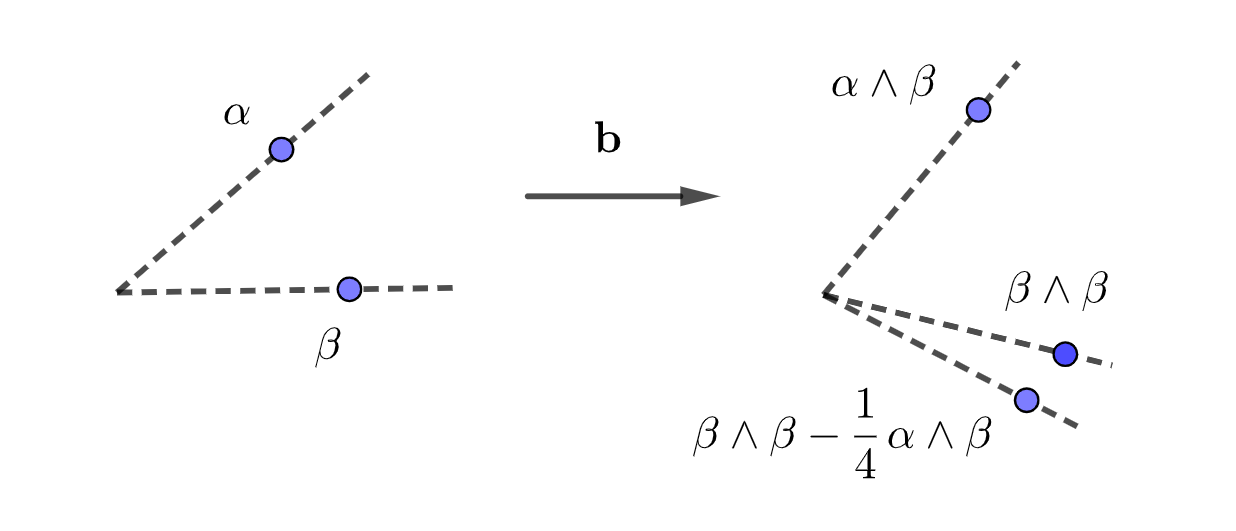}
	\caption{Image of the balanced map of $Y$}
\end{figure}

In Figure 1, the map $\mathbf{b}$ represents the balanced map. The cone on the left denotes the K\"ahler cone $\mathcal{K}(Y)$ of $Y$, which is an open convex cone bounded by the classes $\alpha$ and $\beta$. On the right, the larger cone represents the balanced cone $\mathcal{B}(Y)$ of $Y$, which is an open convex cone bounded by the classes $\alpha \wedge \beta$ and $\beta \wedge \beta - \frac{1}{4} \alpha \wedge \beta$. The open convex cone on the right, bounded by the classes $\alpha \wedge \beta$ and $\beta \wedge \beta$, represents $\mathbf{b}(\mathcal{K}(Y))$. In this way, the difference between $\mathbf{b}(\mathcal{K}(Y))$ and $\mathcal{B}(Y)$ becomes clear.

Since $Y$ is a projective manifold, according to \cite{fu2014relations}, the closure of $\mathcal{B}(Y)$ coincides with the movable cone $\mathcal{M}(Y)$. On the other hand, by \cite{BDPP} and \cite{Toma10}, the movable cone $\mathcal{M}(Y)$ is the dual cone of the pseudo-effective cone $\mathcal{E}(Y)$. Consequently, we obtain the boundary of $\mathcal{B}(Y)$ from the boundary of $\mathcal{E}(Y)$. Clearly, $\alpha$ is a generator of the boundary of $\mathcal{E}(Y)$, and $\alpha$ corresponds to a generator $\alpha \wedge \beta$ of the boundary of $\mathcal{B}(Y)$. To find the generator of the other boundary ray of $\mathcal{B}(Y)$, we need to determine the generator of the other boundary ray of $\mathcal{E}(Y)$.

Previous works on computing the pseudo-effective cone and the movable cone of projective manifolds include   
\cite{GrassModulionCurve}, \cite{StableMapModuli}, \cite{planemuduli}, \cite{canonicalmoduli}, \cite{proModuli}, \cite{canonicalModuli*}, \cite{stablemoduli*}, \cite{quadricmoduli}, and \cite{nefmoduli*},  
where the authors carefully considered specific moduli spaces.   
In particular, \cite{canonicalModuli*} and \cite{nefmoduli*} determine the pseudo-effective cone using the classification of objects in the moduli spaces and subsequently derive the movable cone.   
Meanwhile, \cite{proModuli} considers the case where the numerically effective cone coincides with the pseudo-effective cone.   
Other notable contributions come from \cite{divisorialZariski} and \cite{BlowupPro},  
where \cite{divisorialZariski} characterizes the movable cone of hyper-Kähler manifolds,  
and \cite{BlowupPro} investigates properties of the effective cone of blow-ups of projective spaces.

The rest of the paper is organized as follows. In Section 2, manifold $Y$ is briefly introduced.   
In Section 3, the necessary intersection numbers is computed.   
In Section 4, a generator of the other boundary ray of $\mathcal{E}(Y)$ is determined, and a generator of the other boundary ray of $\mathcal{B}(Y)$ is found by the cone duality.

%第二节，基本的背景与设定 
\section{Small resolution of the quintic conifold}

In this paper, we consider the small resolution of the quintic hypersurface conifold in $\mathbb{P}^4$.  
Let  
\begin{equation*}  
	x = [x_0:x_1:x_2:x_3:x_4]  
\end{equation*}  
be the homogeneous coordinates in $\mathbb{P}^4$, and let $g$ and $h$ be generic quartic homogeneous polynomials in these variables. The singular hypersurface in $\mathbb{P}^4$ defined by $g$ and $h$ is given by  
\begin{equation*}  
	\tilde{Y} = \{x \in \mathbb{P}^4 \mid x_3g(x_0,...,x_4) + x_4h(x_0,...,x_4) = 0\},  
\end{equation*}  
and its singular locus  
\begin{equation*}  
	\mathrm{Sing}(\tilde{Y}) = \{x \in \mathbb{P}^4 \mid x_3 = x_4 = g(x) = h(x) = 0\}  
\end{equation*}  
consists of 16 points (see \cite{rossi2006geometric}).  

Blowing up $\mathbb{P}^4$ along the subvariety  
\begin{equation*}  
	\mathbb{P}^2 = \{x \in \mathbb{P}^4 \mid x_3 = x_4 = 0\}  
\end{equation*}  
yields a submanifold $\hat{\mathbb{P}}^4$ of $\mathbb{P}^4 \times \mathbb{P}^1$, and we denote the blow-up morphism by  
\begin{equation*}  
	\pi: \hat{\mathbb{P}}^4 \rightarrow \mathbb{P}^4.  
\end{equation*}  
Then the subvariety $Y = \overline{\pi^{-1}(\tilde{Y} - \mathrm{Sing}(\tilde{Y}))}$ of $\hat{\mathbb{P}}^4$ is a three-dimensional complex submanifold of $\mathbb{P}^4(x) \times \mathbb{P}^1(y)$, where $y = [y_0:y_1]$ are the homogeneous coordinates of $\mathbb{P}^1$, and $\pi$ induces a small resolution of $\tilde{Y}$  
\begin{equation*}  
	\pi|_{Y}: Y \rightarrow \tilde{Y}.  
\end{equation*}  
Note that  
\begin{equation}\label{definition_Y}  
	Y = \{(x,y) \in \mathbb{P}^4 \times \mathbb{P}^1 \mid y_0x_4 - y_1x_3 = y_0g(x) + y_1h(x) = 0\}.  
\end{equation}  
The following fact about $Y$ is known (see \cite{rossi2006geometric}).  

\begin{proposition}  
	$Y$ is a three-dimensional Calabi-Yau manifold, and $h^{1,1}(Y) = h^{1,1}(\hat{\mathbb{P}}^4) = h^{1,1}(\mathbb{P}^4 \times \mathbb{P}^1) = 2$.  
\end{proposition}

\section{Computations of intersection numbers}

First, we fix some notation.  

Denote the inclusion map  
\begin{equation*}  
	i: Y \hookrightarrow \mathbb{P}^4 \times \mathbb{P}^1,  
\end{equation*}  
and the projection maps  
\begin{equation*}  
	\begin{aligned}  
		\pi_1 &: \mathbb{P}^4 \times \mathbb{P}^1 \rightarrow \mathbb{P}^4, \\  
		\pi_2 &: \mathbb{P}^4 \times \mathbb{P}^1 \rightarrow \mathbb{P}^1.  
	\end{aligned}  
\end{equation*}  

Let  
\begin{equation*}  
	\begin{aligned}  
		\alpha &:= i^*\tilde{\alpha} := i^*\pi_2^*[\omega_{FS, \, \mathbb{P}^1}], \\  
		\beta &:= i^*\tilde{\beta} := i^*\pi_1^*[\omega_{FS, \, \mathbb{P}^4}],  
	\end{aligned}  
\end{equation*}  
where $\omega_{FS}$ is the K\"ahler form corresponding to the Fubini-Study metric in the respective projective spaces.  

Clearly, $\alpha$ and $\beta$ form a basis for $H^{1,1}(Y, \mathbb{R})$,   
and they are generators of the two boundary rays of the K\"ahler cone $\mathcal{K}(Y)$, respectively.   
Furthermore,  
\begin{equation*}  
	\beta \in \mathcal{E}^\circ(Y), \quad   
	\beta \wedge \beta \in \mathcal{B}(Y), \quad   
	\alpha \wedge \beta \in \partial\mathcal{B}(Y).  
\end{equation*}  

We have the following lemma on intersection numbers.  
\begin{lemma}\label{intersection_number}  
	\begin{equation*}  
		\begin{aligned}  
			&\int_Y \alpha \wedge \beta \wedge \beta = 4, \\  
			&\int_Y \beta \wedge \beta \wedge \beta = 5.  
		\end{aligned}  
	\end{equation*}  
\end{lemma}  

\begin{proof}  
	We begin by making some preliminary preparations.  
	Denote $\mathbb{P}^4 \times \mathbb{P}^1$ by $X$ for brevity.  
	Consider the long exact sequence of cohomology groups induced by the exponential sheaf exact sequence on $X$:  
	\begin{equation*}  
		\dots \rightarrow H^1(X, \mathcal{O}_X) \rightarrow H^1(X, \mathcal{O}_X^*) \stackrel{c_1}{\rightarrow} H^2(X, \mathbb{Z}) \rightarrow H^2(X, \mathcal{O}_X) \rightarrow \dots  
	\end{equation*}  
	Since $H^1(X, \mathcal{O}_X)$ and $H^2(X, \mathcal{O}_X)$ are both zero, the morphism $c_1$ is an isomorphism,  
	and $H^2(X, \mathbb{Z})$ is generated by $\tilde{\alpha}$ and $\tilde{\beta}$.  
	
	Note that $V_1 := [y_0x_4 - y_1x_3 = 0]$ and $V_2 := [y_0g(x) + y_1h(x) = 0]$ are divisors in $X$,  
	therefore the intersection numbers on $Y$ can be computed on $X$.  
	\begin{equation*}  
		\begin{aligned}  
			\int_Y \alpha \wedge \beta \wedge \beta &= \int_X [Y] \wedge \tilde{\alpha} \wedge \tilde{\beta} \wedge \tilde{\beta} \\  
			&= \int_X c_1(\mathcal{O}_X(V_1)) \wedge c_1(\mathcal{O}_X(V_2)) \wedge \tilde{\alpha} \wedge \tilde{\beta} \wedge \tilde{\beta}, \\  
			\int_Y \beta \wedge \beta \wedge \beta &= \int_X [Y] \wedge \tilde{\beta} \wedge \tilde{\beta} \wedge \tilde{\beta} \\  
			&= \int_X c_1(\mathcal{O}_X(V_1)) \wedge c_1(\mathcal{O}_X(V_2)) \wedge \tilde{\beta} \wedge \tilde{\beta} \wedge \tilde{\beta}.  
		\end{aligned}  
	\end{equation*}  
	Thus, we need to compute the first Chern classes of divisors of bihomogeneous polynomials on $X$.  
	
	Let $u(x)$ be a homogeneous polynomial of degree $k$,  
	$v(y)$ be a homogeneous polynomial of degree $l$,  
	and $w(x, y)$ be a bihomogeneous polynomial of degree $k$ in $x$ and $l$ in $y$.  
	Denote  
	\begin{equation*}  
		D_1 := [u(x) = 0], \quad D_2 := [v(y) = 0], \quad D_3 := [w(x, y) = 0].  
	\end{equation*}  
	Observe that $X$ has the following affine open cover  
	\begin{equation*}  
		\{W_{ij} := (x_i \neq 0) \cap (y_j \neq 0)\}_{i=0,1,2,3,4; \, j=0,1}.  
	\end{equation*}  
	
	By examining the transition functions of the line bundle $\mathcal{O}_X(D_1)$ over this open cover,  
	we see that $\mathcal{O}_X(D_1)$ is isomorphic to $\pi_1^*(\mathcal{O}_{\mathbb{P}^4}(k))$.  
	Indeed, the transition functions of $\mathcal{O}_X(D_1)$ are  
	\begin{equation*}  
		\{g_{ij,i'j'} = \frac{x_{i'}^k}{x_i^k}\}_{ij,i'j'}.  
	\end{equation*}  
	Hence, by the properties of Chern classes,  
	\begin{equation*}  
		c_1(\mathcal{O}_X(D_1)) = k\tilde{\beta}.  
	\end{equation*}  
	Similarly, $\mathcal{O}_X(D_2)$ is isomorphic to $\pi_2^*(\mathcal{O}_{\mathbb{P}^1}(l))$,  
	and thus  
	\begin{equation*}  
		c_1(\mathcal{O}_X(D_2)) = l\tilde{\alpha}.  
	\end{equation*}  
	Likewise, the transition functions of $\mathcal{O}_X(D_3)$ are  
	\begin{equation*}  
		\{\hat{g}_{ij,i'j'} = \frac{x_{i'}^k y_{j'}^l}{x_i^k y_j^l}\}_{ij,i'j'},  
	\end{equation*}  
	and $\mathcal{O}_X(D_3)$ is isomorphic to $\mathcal{O}_X(D_1 + D_2)$. Therefore,  
	\begin{equation*}  
		c_1(\mathcal{O}_X(D_3)) = c_1(\mathcal{O}_X(D_1)) + c_1(\mathcal{O}_X(D_2)) = l\tilde{\alpha} + k\tilde{\beta}.  
	\end{equation*}  
	From the definitions of $V_1$ and $V_2$, we obtain  
	\begin{equation*}  
		\begin{aligned}  
			c_1(\mathcal{O}_X(V_1)) &= \tilde{\alpha} + \tilde{\beta}, \\  
			c_1(\mathcal{O}_X(V_2)) &= \tilde{\alpha} + 4\tilde{\beta}.  
		\end{aligned}  
	\end{equation*}  
	
	Now we compute the intersection numbers from the lemma.  
	\begin{equation*}  
		\begin{aligned}  
			\int_Y \alpha \wedge \beta \wedge \beta &= \int_X [Y] \wedge \tilde{\alpha} \wedge \tilde{\beta} \wedge \tilde{\beta} \\  
			&= \int_X (\tilde{\alpha} + \tilde{\beta}) \wedge (\tilde{\alpha} + 4\tilde{\beta}) \wedge \tilde{\alpha} \wedge \tilde{\beta} \wedge \tilde{\beta} \\  
			&= 4 \int_X \tilde{\beta}^4 \wedge \tilde{\alpha} \\  
			&= 4.  
		\end{aligned}  
	\end{equation*}  
	\begin{equation*}  
		\begin{aligned}  
			\int_Y \beta \wedge \beta \wedge \beta &= \int_X [Y] \wedge \tilde{\beta} \wedge \tilde{\beta} \wedge \tilde{\beta} \\  
			&= \int_X (\tilde{\alpha} + \tilde{\beta}) \wedge (\tilde{\alpha} + 4\tilde{\beta}) \wedge \tilde{\beta} \wedge \tilde{\beta} \wedge \tilde{\beta} \\  
			&= 5 \int_X \tilde{\beta}^4 \wedge \tilde{\alpha} \\  
			&= 5.  
		\end{aligned}  
	\end{equation*}  
\end{proof}

%第四节，得出主要结果
\section{Boundary of the balanced cone}

Let $\alpha$ be a generator of the boundary of $\mathcal{E}(Y)$.  
We need to find a generator of another boundary ray of $\mathcal{E}(Y)$.  
Since $\beta$ is a big class, the generator of the other boundary ray of $\mathcal{E}(Y)$ has the form  
\begin{equation*}  
	\xi = \beta - c\alpha,  
\end{equation*}  
where $c$ is a positive constant to be determined.  

Since the Hodge conjecture holds for $(1,1)$-classes, if we can find the classes corresponding to all effective divisors of $Y$, we can determine the boundary of the pseudo-effective cone.  
Furthermore, since every effective divisor class is a non-negative linear combination of prime divisor classes, it suffices to determine all prime divisor classes on $Y$.  

From the defining equation (\ref{definition_Y}) of $Y$, we know that $Y$ has an affine open subset  
\begin{equation*}  
	Y \cap W_{00} \cong \{(\tilde{x},\tilde{y}) \in \mathbb{C}^5 \mid \tilde{x}_4 - \tilde{y}_1\tilde{x}_3 = g(\tilde{x}) + \tilde{y}_1h(\tilde{x}) = 0\}.  
\end{equation*}  
Here, $\tilde{x} = (\tilde{x}_1, \tilde{x}_2, \tilde{x}_3, \tilde{x}_4)$ and $\tilde{y} = (\tilde{y}_1)$ represent the non-homogeneous coordinates of the homogeneous coordinates $x$ and $y$ respectively, on $W_{00}$.  

In particular, the following divisor relations hold on $Y \cap W_{00}$:  
\begin{equation*}  
	\begin{aligned}  
		[\tilde{x}_4 = 0] &= [\tilde{y}_1 = 0] + [\tilde{x}_3 = 0], \\  
		[g(\tilde{x}) = 0] &= [\tilde{y}_1 = 0] + [h(\tilde{x}) = 0].  
	\end{aligned}  
\end{equation*}  

Note that $E_1 = [x_3 = x_4 = 0]$ and $E_2 = [g(x) = h(x) = 0]$ are the closures in $Y$ of the divisors $[\tilde{x}_3 = 0]$ and $[h(\tilde{x}) = 0]$ respectively.  
By the Jacobian criterion and the genericity of $g$ and $h$, $[\tilde{x}_3 = 0]$ and $[h(\tilde{x}) = 0]$ are both smooth and irreducible, hence they are prime divisors in $Y \cap W_{00}$.  
Thus, their closures $E_1$ and $E_2$ are also prime divisors in $Y$, and we have  
\begin{equation*}  
	\begin{aligned}  
		[x_4 = 0] &= [y_1 = 0] + E_1, \\  
		[g(x) = 0] &= [y_1 = 0] + E_2.  
	\end{aligned}  
\end{equation*}  

From this, we obtain  
\begin{equation*}  
	\begin{aligned}  
		c_1(\mathcal{O}_Y(E_1)) &= \beta - \alpha, \\  
		c_1(\mathcal{O}_Y(E_2)) &= 4\beta - \alpha.  
	\end{aligned}  
\end{equation*}  

\begin{lemma}\label{primeDivisorClass}  
	Every prime divisor in $Y$ corresponds to a first Chern class lying in the closed convex cone $\mathrm{span}_{\mathbb{R}_{\geqslant 0}}\{\alpha, \beta - \alpha\}$.  
\end{lemma}  

\begin{proof}  
	Let $D_0$ be a prime divisor. Since we only need to find prime divisor classes different from $E_1$, assume without loss of generality that $D_0 \neq E_1$. Denote  
	\begin{equation*}  
		c_1(\mathcal{O}_Y(D_0)) = a_1\alpha + a_2\beta.  \end{equation*}  
	Since $D_0$ and $E_1$ are distinct prime divisors, for a holomorphic line bundle $A$ whose first Chern class is ample, the following inequality holds:  
	\begin{equation}\label{inequality}  
		\int_Y{c_1(\mathcal{O}_Y(D_0)) \wedge c_1(\mathcal{O}_Y(E_1)) \wedge c_1(A)} \geqslant 0.  
	\end{equation}  
	Let $c_1(A) = C_1\alpha + C_2\beta$, where $C_1 > 0$ and $C_2 > 0$. The left-hand side of inequality (\ref{inequality}) is computed as follows:  
	\begin{equation}\label{LHSRHS}  
		\begin{aligned}  
			\mathrm{LHS} &= \int_Y{(a_1\alpha + a_2\beta) \wedge (-\alpha + \beta) \wedge (C_1\alpha + C_2\beta)} \\  
			&= a_2C_2\int_Y{\beta \wedge \beta \wedge \beta} + (a_1C_2 - a_2C_2 + a_2C_1)\int_Y{\alpha \wedge \beta \wedge \beta} \\  
			&= a_2C_2 + 4a_2C_1 + 4a_1C_2.  
		\end{aligned}  
	\end{equation}  
	Without loss of generality, take  
	\begin{equation}\label{take_special_ample}  
		C_1 = \frac{3}{4}C_2 > 0.  
	\end{equation}  
	Substituting (\ref{take_special_ample}) into (\ref{LHSRHS}), and using inequality (\ref{inequality}), we obtain  
	\begin{equation*}  
		4(a_1 + a_2)C_2 \geqslant 0.  
	\end{equation*}  
	Thus,  
	\begin{equation*}  
		a_1 + a_2 \geqslant 0.  
	\end{equation*}  
	On the other hand, since $D_0$ is an effective divisor, and $\alpha \in \partial\mathcal{E}(Y)$, $\beta \in \mathcal{E}^\circ(Y)$, it follows that $a_2 \geqslant 0$. Therefore,  
	\begin{equation*}  
		\begin{aligned}  
			c_1(\mathcal{O}_Y(D_0)) &= a_1\alpha + a_2\beta \\  
			&= a_2(\beta - \alpha) + (a_1 + a_2)\alpha \in \mathrm{span}_{\mathbb{R}_{\geqslant 0}}\{\alpha, \beta - \alpha\}.  
		\end{aligned}  
	\end{equation*}  
\end{proof}

By the previous lemma, we have
\begin{proposition}  
	$\mathcal{E}(Y)$ coincides with the closed convex cone $\mathrm{span}_{\mathbb{R}_{\geqslant 0}}\{\alpha, \beta - \alpha\}$.  
\end{proposition}

Now we need the following lemma.

\begin{lemma}\label{dual_cone_bdry_relation}  
	Let $C$ be an open convex cone with vertex at $0$ in an $n$-dimensional real vector space $V$, and let $C^{\vee}$ denote its dual convex cone. Then, for any nonzero $x \in \partial C$, there exists a nonzero $G \in \partial (C^{\vee})$ such that  
	\begin{equation*}  
		G(x) = 0.  
	\end{equation*}  
\end{lemma}  

\begin{proof}  
	Assume $x \in \partial C$ is nonzero. Since $C$ is an open convex cone, we apply the Hahn-Banach Separation Theorem (Theorem 3.4 in \cite{rudin1973functional}) to $C$ and $\mathrm{span}_{\mathbb{R}}\{x\}$. Consequently, there exists a linear map $F: V \rightarrow \mathbb{R}$ and a real number $\gamma$ such that for any $\hat{x} \in \mathrm{span}_{\mathbb{R}}\{x\}$ and any $y \in C$,  
	\begin{equation*}  
		F(y) < \gamma \leq F(\hat{x}).  
	\end{equation*}  
	In this case, $F(\hat{x}) = \gamma = 0$. Setting $G = -F$, it follows from the properties of $F$ that $G \in C^{\vee}$.  
	
	We need to show that $G \in \partial (C^{\vee})$. Let $\{u_1, \ldots, u_{n-1}, x\}$ be a basis for $V$. For each $\varepsilon > 0$, define a linear map $H_{\varepsilon}: V \rightarrow \mathbb{R}$ such that for any $(a_1, \ldots, a_n) \in \mathbb{R}^n$,  
	\begin{equation*}  
		H_{\varepsilon}\left(\sum_{i=1}^{n-1}{a_iu_i} + a_nx\right) = -\varepsilon a_n.  
	\end{equation*}  
	Thus, for any $\varepsilon > 0$,  
	\begin{equation*}  
		G + H_{\varepsilon} \notin C^{\vee}.  
	\end{equation*}  
	On the other hand,  
	\begin{equation*}  
		G = \lim_{\varepsilon \rightarrow 0} (G + H_{\varepsilon}) \in C^{\vee}.  
	\end{equation*}  
	Therefore, $G \in \partial (C^{\vee})$ and, by the above construction, $G$ is nonzero.  
\end{proof}

Now, by the conclusions from \cite{BDPP}, \cite{fu2014relations}, and \cite{Toma10}, we derive the other boundary ray of $\mathcal{B}(Y)$.  

Indeed, by Lemma \ref{dual_cone_bdry_relation} and the fact that $\overline{\mathcal{B}}(Y)$ is the dual convex cone of $\mathcal{E}^\circ(Y)$, it suffices to consider the intersection numbers of boundary elements of $\mathcal{E}^\circ(Y)$ with the basis of the $(2,2)$-class.  

From  
\begin{equation*}  
	\int_Y \alpha \wedge (\alpha \wedge \beta) = 0,  
\end{equation*}  
and since $\alpha \wedge \beta$ is semi-positive, we reconfirm that $\alpha \wedge \beta$ generates a boundary ray of $\mathcal{B}(Y)$. Given that $\beta \wedge \beta \in \mathcal{B}(Y)$, the generator of the other boundary ray of $\mathcal{B}(Y)$ has the form  
\begin{equation*}  
	\beta \wedge \beta - c \alpha \wedge \beta,  
\end{equation*}  
where the positive constant $c$ can be determined now. 

By Lemma \ref{dual_cone_bdry_relation} and Lemma \ref{intersection_number},  
\begin{equation*}  
	\begin{aligned}  
		0 &= \int_Y (\beta - \alpha) \wedge (\beta \wedge \beta - c \alpha \wedge \beta) \\  
		&= \int_Y \beta \wedge \beta \wedge \beta - \int_Y \alpha \wedge \beta \wedge \beta - c \int_Y \beta \wedge \alpha \wedge \beta \\  
		&= 5 - 4 - 4c \\  
		&= 1 - 4c.  
	\end{aligned}  
\end{equation*}  
Thus, we obtain $c = \frac{1}{4}$, and therefore, the generator of the other boundary ray of $\mathcal{B}(Y)$ is  
\begin{equation*}  
	\beta \wedge \beta - \frac{1}{4} \alpha \wedge \beta.  
\end{equation*}  

This completes the proof of the main theorem.

\bibliographystyle{abbrv}

\end{document}